\documentclass[11pt]{amsart}  
  
\usepackage{mathrsfs}
\usepackage{amscd}
\usepackage{amsmath}
\usepackage{amssymb}
\usepackage{amsthm}
\usepackage{epsf}
\usepackage{verbatim}
\usepackage{enumitem}
\usepackage{mathtools}
\usepackage{xcolor}
\usepackage{array}
\usepackage{bm}
\usepackage{bbm}
\usepackage{graphicx}
\input epsf.tex
\usepackage{amssymb}

\usepackage[left=3cm, right=3cm, top=2cm]{geometry}

\newtheorem{theorem}{Theorem}

\newtheorem{corollary}[theorem]{Corollary}
\theoremstyle{definition}

\newtheorem{proposition}[theorem]{Proposition}
\theoremstyle{definition}

\theoremstyle{definition}

\theoremstyle{definition}
\newtheorem{definition}[theorem]{Definition}

\newcommand{\Nat}{\mathbb{N}}
\newcommand{\T}{\mathcal{T}}

\newcommand{\U}{\mathcal{U}}
\newcommand{\V}{\mathcal{V}}

\newcommand{\N}{\mathcal{N}}

\newcommand{\diam}{\textup{diam}}

\title{Borel families of games}

\author{Alexander Kastner and Clark Lyons}

\begin{document}

\maketitle

A \emph{family of games} is a set $B \subseteq X \times \N$, where $\N$ is the Baire space. We think of each vertical section $B_x \subseteq \N$ as the payoff set for player II. In this note, we give an elementary proof of the following theorem, which is related to recent work on homomorphism graphs in \cite{HomomorphismGraphs}. The theorem probably follows from results in \cite{SchillingVaught}, but we give a much more streamlined proof. The result can also be obtained as a consequence of Feng, Magidor, and Woodin's work on universally Baire set of reals \cite{FMW}. We refer to Kechris's book \cite{Kechris} for definitions of all undefined concepts.

\begin{theorem} \label{maincor}
Let $X$ be a Polish space, and suppose that $B \subseteq X \times \N$ is a Borel family of games. Then
\[ W = \{x \in X: \text{player II has a winning strategy in $B_x$}\}\]
is Baire measurable and universally measurable. In the case where $X = [\Nat]^{\aleph_0}$, the set $W$ is also completely Ramsey.
\end{theorem}

The following game and theorem, in the case when $(X, \T)$ is Polish and the $U_i, V_i$ come from a countable basis, are apparently due to Solovay in unpublished notes.

\begin{definition}
Let $(X, \T)$ be a Choquet space, and let $d$ be a metric whose open balls are in $\T$. Suppose that $\sigma_C$ is a winning strategy for player II in the Choquet game for $X$. We may assume that $\sigma_C$ only depends on the most recent move of player I. If $B \subseteq X \times \N$, we define the game $\mathcal{G}(X, B)$ as follows: \\
$\begin{matrix}
&\textup{I} \quad & (U_0, m_0) & & (U_1, m_1) \\
& &&&&& $\dots$ \\
&\textup{II} & & (V_0, n_0) & & (V_1, n_1) &
\end{matrix}$ \\
where the $m_i, n_i$ are natural numbers, $U_i, V_i$ are open sets in $\T$, $\diam_d(U_i), \diam_d(V_i) < 2^{-i}$, $V_i \subseteq U_i$, and $U_{i+1} \subseteq \sigma_C(V_i)$. This ensures that $\bigcap_i U_i = \bigcap_i V_i = \{x\}$ is a singleton. Player~II wins iff $(x, (m_0, n_0, m_1, n_1, \dots)) \in B$. Note that we can always choose to restrict the $U_i$ and $V_i$ to any weak basis for $(X, \T)$.
\end{definition}

\begin{theorem} \label{mainlem}
Let $(X, \T)$ be a Choquet space, let $d$ be a metric whose open balls are in $\T$, and let $B \subseteq X \times \N$.
\begin{enumerate}[label=(\alph*)]
\item If player II has a winning strategy in $\mathcal{G}(X, B)$, then
\[ W = \{x \in X: \text{player II has a winning strategy in $B_x$}\}\]
is comeager.
\item If player I has a winning strategy in $\mathcal{G}(X, B)$, then
\[ L = \{x \in X: \text{player I has a winning strategy in $B_x$}\}\]
is comeager in some nonempty open set $U \in \T$.
\end{enumerate}
\end{theorem}
\begin{proof}
(a) Suppose that $\sigma$ is a winning strategy for player II. For each $m_0$, Zorn's lemma yields a collection $\U_{m_0}$ of open sets $U_0 \in \T$ such that
\[ \V_{m_0} := \{\sigma(U_0, m_0)_0: U_0 \in \U_{m_0}\}\]
is a disjoint collection of open sets in $\T$ whose union is dense in $(X, \T)$. Now for each $U_0 \in \U_{m_0}$ and $m_1$, Zorn's lemma yields a collection $\U_{U_0, m_0, m_1}$ of open sets $U_1 \in \T$ such that
\[ \V_{U_0, m_0, m_1} := \{\sigma(U_0, m_0, U_1, m_1)_0: U_1 \in \U_{U_0, m_0, m_1}\} \]
is a disjoint collection of open sets in $\T$ whose union is dense in $\sigma(U_0,m_0)_0$. In particular, for each fixed $m_0, m_1$, the collection
\[ \V_{m_0, m_1} := \bigcup_{U_0 \in \U_0} \V_{U_0, m_0, m_1} \]
has dense union in $(X, \T)$. By continuing the construction above, we get collections $\V_{m_0, \dots, m_k}$ of disjoint open sets in $\T$, which in particular have dense union in $(X, \T)$. We claim that the dense $G_\delta$
\[ \bigcap_{k \in \Nat} \bigcap_{m_0, \dots, m_k} \left( \bigcup \V_{m_0, \dots, m_k} \right) \]
is contained in $W$. Fix $x$ in this set. We define a winning strategy for player II for the game $B_x$. If player I first plays $m_0$, let $U_0$ be the unique element of $\U_{m_0}$ such that $x \in \sigma(U_0, m_0)_0$. Player II should play $\sigma(U_0, m_0)_1$. Now suppose that player I's next move is $m_1$. Let $U_1$ be the unique element of $\U_{U_0, m_0, m_1}$ such that $x \in \sigma(U_0, m_0, U_1, m_1)_0$. Player II should then play $\sigma(U_0, m_0, U_1, m_1)_1$. Keep going in this way. By construction, we will have
\[ (x, (m_0, n_0, m_1, n_1, \dots)) \in B,\]
so this describes a winning strategy for player II in $B_x$. The proof of part (b) is entirely analogous.
\end{proof}

\begin{proposition}\label{BM characterization}
Let $X$ be a Baire space, and let $A \subseteq X$. Then $A$ is Baire measurable iff for every nonempty open set $U \subseteq X$, either $A$ is comeager in $U$ or there exists a nonempty open $V \subseteq U$ such that $A$ is meager in $V$.
\end{proposition}
\begin{proof}
The forward direction is just the well-known Baire alternative applied to $U$. For the backward direction, consider the open set
\[ U(A^c) := \bigcup \{U: \text{$U$ is a nonempty open set in which $A$ is meager}\}.\]
By Theorem 8.29 in \cite{Kechris}, $A$ is meager in $U(A^c)$. The space $X$ can be partitioned as
\[ X = U(A^c) \sqcup \partial U(A^c) \sqcup (X \setminus \overline{U(A^c)}).\]
The boundary $\partial U(A^c)$ is meager, and we claim that $A$ is comeager in the open set $X \setminus \overline{U(A^c)}$. Otherwise, there is a nonempty open set $V \subseteq X \setminus \overline{U(A^c)}$ in which $A$ is meager. But this would contradict the definition of $U(A^c)$. Let's summarize: $A$ is meager in $U(A^c)$, $\partial U(A^c)$ is meager, and $A$ is comeager in $X \setminus \overline{U(A^c)}$. It follows that $A$ is Baire measurable.
\end{proof}

% We even have a bit more: the assumptions in the following theorem imply that $W \cup L$ is comeager in $X$. Just modify the proof of the BM characterization slightly.
\begin{corollary}
Let $(X, \T)$ be a Choquet space, let $d$ be a metric whose open balls are in $\T$, and let $B \subseteq X \times \N$. Suppose that for every open set $U \in \T$, the game $\mathcal{G}(U, B)$ is determined. Then $W$ is Baire measurable in $(X, \T)$.
\end{corollary}

\begin{theorem}[General Borel Determinacy]
Let $A$ be a discrete topological space (possibly uncountable). Then any Borel $B \subseteq A^\Nat$ corresponds to a determined game.
\end{theorem}

\begin{theorem} \label{main}
Let $X$ be a Polish space with compatible metric $d$, and suppose that $B \subseteq X \times \N$ is a Borel family of games. Suppose that $\T$ is a Choquet topology on $X$ such that the open $d$-balls are in $\T$. Then
\[ W := \{x \in X: \text{player II has a winning strategy in $B_x$}\}\]
is Baire measurable in $(X, \T)$.
\end{theorem}
\begin{proof}
We need to show that the game $\mathcal{G}(U, B)$ is determined for every $U \in \T$. So fix $U \in \T$, and let $T$ be the tree of legal positions in $\mathcal{G}(U, B)$. Consider the map
\[ \begin{aligned}
\varphi: [T] &\rightarrow X \times \N \\
\left((U_0, m_0), (V_0, n_0), (U_1, m_1), \dots\right) &\mapsto (x, (m_0, n_0, m_1, \dots))
\end{aligned}\]
where $\{x\} = \bigcap_i U_i = \bigcap_i V_i$. If we can show that $\varphi: [T] \rightarrow (X, d) \times \N$ is continuous (where $[T]$ is viewed as a closed subspace of a product of uncountable discrete spaces), then this will imply that the payoff set for player II in $\mathcal{G}(U, B)$ is Borel as a subset of $[T]$. Hence, the General Borel Determinacy Theorem would imply that $\mathcal{G}(U, B)$ is determined.

Suppose that $\left((U_0, m_0), (V_0, n_0), (U_1, m_1), \dots\right) \in [T]$, and let
\[ \varphi\left((U_0, m_0), (V_0, n_0), (U_1, m_1), \dots\right) = (x, (m_0, n_0, m_1, \dots)).\]
A basic open neighborhood of $(x, (m_0, n_0, m_1, \dots))$ will have the form
\[ B_d(x; 2^{-k}) \times N_{m_0, n_0, \dots, m_k, n_k}.\]
Our requirement that $\diam_d(U_{k+1}), \diam_d(V_{k+1}) < 2^{-(k+1)}$ ensures that
\[ \varphi[N_{(U_0, m_0), (V_0, n_0), \dots, (U_{k+1}, m_{k+1}), (V_{k+1}, m_{k+1})}] \subseteq B_d(x; 2^{-k}) \times N_{m_0, n_0, \dots, m_k, n_k}.\]
This verifies that $\varphi: [T] \rightarrow (X, d) \times \N$ is a continuous function, and we are done.
\end{proof}

Our main theorem, Theorem~\ref{maincor}, now easily follows.

\begin{proof}[Proof of Theorem~\ref{maincor}]
We apply Theorem~\ref{main} for different choices of $\T$. To show $W$ is Baire measurable, take $\T$ to be the given Polish topology. To show $W$ is universally measurable, take $\T$ to be the density topology (after reducing to the case where the measure is Lebesgue measure on $(0,1)$). To show $W$ is completely Ramsey, take $\T$ to be the Ellentuck topology.
\end{proof}

We also remark that the $\mathcal{G}$ game and Theorem~\ref{mainlem} can be used to prove the classical Kuratowski–Ulam theorem.

\begin{theorem}
If $X$ and $Y$ are Polish spaces and $A\subseteq X\times Y$ is comeager, then $A_x$ is comeager in $Y$ for a comeager set of $x\in X$.
\end{theorem}

\begin{proof}
Consider $B\subseteq X\times\mathcal{N}$, where for every $x\in X$, the game $B_x\subseteq\mathcal{N}$ is the Banach Mazur game for $A_x\subseteq Y$. Then $\mathcal{G}(X,B)$ is exactly the Banach Mazur game for $A\subseteq X\times Y$ in the product topology.

Since $A\subseteq X\times Y$ is comeager, player II has a winning strategy for $\mathcal{G}(X,B)$. By Theorem~\ref{mainlem}, player II has a winning strategy for $B_x\subseteq\mathcal{N}$ for a comeager set of $x\in X$. And therefore $A_x$ is comeager in $Y$ for comeager many $x\in X$.
\end{proof}

The above proof can be adapted to prove a generalization of the Kuratowski-Ulam theorem to the setting where $Y$ is a topological space with a $\pi$-basis of cardinality $\kappa$, and $X$ is a topological space with the property that the intersection of $\kappa$ many comeager subsets of $X$ is comeager.

\end{document}